\documentclass[11pt, a4paper, oneside]{amsart}

\usepackage{newtxtext, newtxmath}

\usepackage{amsmath, amsthm} 
\usepackage{mathrsfs}     
\usepackage{verbatim} 
\usepackage{graphicx}             

\usepackage[normalem]{ulem} 
\usepackage{enumitem}       
\usepackage{fancyhdr}       

\theoremstyle{plain}
\newtheorem*{thm}{Theorem}
\newtheorem{lem}{Lemma}
\theoremstyle{definition} 
\newtheorem*{defn*}{Definition}
 
\newtheorem*{ack}{Acknowledgements}
\theoremstyle{remark}
\newtheorem*{rem*}{Remark}

 \newcommand{\V}[1]{\Vert #1 \Vert} 
\newcommand{\Czero}[1]{C_0(#1)}
\newcommand{\CzeroPlus}[1]{C_0^+(#1)}
\newcommand{\TildeT}{\tilde{T}}

\begin{document}

\title[Additive and multiplicative maps]
{Additive and multiplicative maps in norm on the positive cone of continuous function algebras}

\author[T. Miura]{Takeshi Miura}
\address[T. Miura]
{Department of Mathematics,
Faculty of Science,
Niigata University,
Niigata 950-2181, Japan}
\email{miura@math.sc.niigata-u.ac.jp}

\author[N.~Shibata]{Natsumi Shibata}
\address[N. Shibata]{Graduate School of Science and Technology,
Niigata University,
Niigata 950-2181, Japan}
\email{f25a056h@mail.cc.niigata-u.ac.jp}


% --- Abstract ---
\begin{abstract}
Let $X$ and $Y$ be locally compact Hausdorff spaces. We denote by $\CzeroPlus{X}$ the positive cone of all real-valued continuous functions on $X$ vanishing at infinity. In this paper, we consider a bijection $T\colon \CzeroPlus{X} \to \CzeroPlus{Y}$ satisfying the following two norm conditions for all $f, g \in \CzeroPlus{X}$:
\[
    \V{T(f+g)} = \V{T(f)+T(g)},\qquad
    \V{T(f \cdot g)} = \V{T(f) \cdot T(g)}.
\]
The main result of this paper is that such a map $T$ is a composition operator of the form $T(f) = f \circ \tau$, induced by a homeomorphism $\tau\colon Y \to X$.
\end{abstract}

\maketitle

\section{Introduction and main results}

Throughout this paper, let $X$ and $Y$ be locally compact Hausdorff spaces.
We denote by $\Czero{X}$ the Banach algebra of all real-valued continuous functions on $X$ vanishing at infinity, equipped with the supremum norm $\V{\cdot}$.
Note that if $X$ is compact, $\Czero{X}$ coincides with $C(X)$, the Banach algebra of all real-valued continuous functions on $X$.
The classical Banach--Stone Theorem \cite{BanachStone,Stone} asserts that for compact Hausdorff spaces $X$ and $Y$, $C(X)$ is linearly isometric to $C(Y)$ if and only if $X$ is homeomorphic to $Y$.
It is well known that this result extends to locally compact Hausdorff spaces $X$ and $Y$ (see, e.g., \cite{Wang95}).
Specifically, every surjective linear isometry $U\colon \Czero{X} \to \Czero{Y}$ is of the form
$$
U(f)(y) = \alpha(y) f(\tau(y)) \quad (f \in \Czero{X},\, y \in Y),
$$
where $\alpha\colon Y \to \{-1, 1\}$ is a continuous function and $\tau\colon Y \to X$ is a homeomorphism.
The Banach--Stone theorem states that the metric structure of $\Czero{X}$ determines the topological structure of $X$.

It is a natural question whether the algebraic structures of $\Czero{X}$ determine the topological structure of $X$.
Let $T\colon \Czero{X} \to \Czero{Y}$ be a bijective map.
We call $T$ a \textit{ring isomorphism} if $T$ satisfies the following conditions for all $f, g \in \Czero{X}$:
$$
T(f+g) = T(f) + T(g),\qquad
T(f \cdot g) = T(f) \cdot T(g).
$$
Gelfand and Kolmogoroff \cite{GK} proved that for compact Hausdorff spaces $X$ and $Y$, every ring isomorphism $T\colon C(X) \to C(Y)$ induces a homeomorphism between $X$ and $Y$.
Therefore, the algebraic structure of $C(X)$ determines the topological structure of $X$ by the Gelfand--Kolmogoroff theorem.

Recently, Dong, Lin, and Zheng \cite{GK-orig} weakened the algebraic condition of the Gelfand--Kolmogoroff theorem by introducing the notion of a \textit{ring isomorphism in norm} between $C(X)$ and $C(Y)$.
We call $T\colon C(X) \to C(Y)$ a ring isomorphism in norm if it is a bijective map that satisfies the following conditions for all $f, g \in C(X)$:
$$
\V{T(f+g)} = \V{T(f) + T(g)},\qquad
\V{T(f \cdot g)} = \V{T(f) \cdot T(g)}.
$$
The authors of \cite{GK-orig} proved that such a map $T$ is also a weighted composition operator induced by a homeomorphism.

In this paper, we consider the positive cone $\CzeroPlus{X}$ of $\Czero{X}$ defined by
$$
\CzeroPlus{X} = \{ f \in \Czero{X} \mid f(x) \ge 0 \text{ for all } x \in X \}.
$$
We investigate a bijective map
$T\colon \CzeroPlus{X} \to \CzeroPlus{Y}$ satisfying the above two norm conditions for all $f, g \in \CzeroPlus{X}$.
Our main theorem is as follows:

\begin{thm}
Let $X$ and $Y$ be locally compact Hausdorff spaces.
Let $T\colon \CzeroPlus{X} \to \CzeroPlus{Y}$ be a ring isomorphism in norm, that is, bijective map that, for all $f, g \in \CzeroPlus{X}$, satisfies
\begin{align}
    \V{T(f+g)} &= \V{T(f)+T(g)}, \label{eq:1} \\
    \V{T(f \cdot g)} &= \V{T(f) \cdot T(g)}. \label{eq:2}
\end{align}
Then, there exists a homeomorphism $\tau\colon Y \to X$ such that
$$
T(f)(y) = f(\tau(y)) \quad (f \in \CzeroPlus{X},\, y \in Y).
$$
\end{thm}

This result extends the theorem of Dong, Lin, and Zheng \cite{GK-orig} to the setting of positive cones  $\CzeroPlus{X}$ and $\CzeroPlus{Y}$ for locally compact spaces.

The proof is divided into several lemmas.
The core idea is to extend the map $T$ from the positive cone $\CzeroPlus{X}$ to the entire space $\Czero{X}$.
To this end, we consider the positive and negative parts of functions in $\Czero{X}$.

\section{Proof of Main Theorem}

Throughout this section, let $T\colon \CzeroPlus{X} \to \CzeroPlus{Y}$ be a bijective map satisfying conditions \eqref{eq:1} and \eqref{eq:2}.
Hirota \cite{Hirota} proved that any surjective map from $\CzeroPlus{X}$ to $\CzeroPlus{Y}$ satisfying condition \eqref{eq:1} is additive and positively homogeneous.
We state this result as Lemma~\ref{lem:additive} below; see \cite{Hirota} for the proof.

% --- Lemma 1 ---
\begin{lem}[{Hirota~\cite[Theorem~1.1]{Hirota}}]
\label{lem:additive}
The map $T$ is additive and positively homogeneous;
that is, $T(f+g)=T(f)+T(g)$ and $T(rf)=rT(f)$
for all $f, g \in \CzeroPlus{X}$ and all $r \ge 0$.
\end{lem}

% --- Lemma 2 の導入文 ---
The additivity established in Lemma~\ref{lem:additive} implies the following property.

% --- Lemma 2 ---
\begin{lem} \label{lem:order_preserving}
The map $T$ is order-preserving; that is, if $f, g \in \CzeroPlus{X}$ and $f \le g$, then $T(f) \le T(g)$.
\end{lem}

% --- Lemma 2 の証明 ---
\begin{proof}
Let $f, g \in \CzeroPlus{X}$ satisfy $f \le g$.
Define $h = g - f$. Then $h \in \CzeroPlus{X}$
with $g = f + h$.
By the additivity of $T$ from Lemma~\ref{lem:additive},
$$
T(g) = T(f) + T(h).
$$
Since $T(h) \in \CzeroPlus{Y}$, it follows that
$T(g) - T(f) = T(h) \ge 0$.
Therefore, $T(f) \le T(g)$.
\end{proof}

% --- Lemma 3 の導入文  ---
We now construct the extension map $\TildeT$ from $\Czero{X}$ to $\Czero{Y}$.
For any $f \in \Czero{X}$, we set $f_+ = \max\{f, 0\}$ and $f_- = \max\{-f, 0\}$.
Then we can write $f = f_+ - f_-$.
We define the map $\TildeT\colon$ $\Czero{X} \to \Czero{Y}$ by
\begin{equation}
 \TildeT(f) = T(f_+) - T(f_-). \label{eq:TildeT_def}
\end{equation}
If $f \in \CzeroPlus{X}$, then $f_+ = f$ and $f_- = 0$.
Since $T(0) = 0$ by Lemma~\ref{lem:additive}, we have
$\TildeT(f) = T(f)$
for all $f \in \CzeroPlus{X}$.
Thus, $\TildeT$ and $T$ coincide on $\CzeroPlus{X}$.
The following lemma establishes the linearity of this extension.

%--- Lemma 3 ---
\begin{lem} \label{lem:extension}
The extension map $\TildeT$ defined in \eqref{eq:TildeT_def} is linear.
\end{lem}

%---Lemma 3 の証明 ---
\begin{proof}
%--- additivity ---
\textbf{Additivity}:
Let $f, g \in \Czero{X}$ and let $h = f + g$.
We show that $\TildeT(h) = \TildeT(f) + \TildeT(g)$.
We have the identity
$$
h_+ - h_- = h = f + g = (f_+ - f_-) + (g_+ - g_-).
$$
Rearranging terms, we have
$$
h_+ + f_- + g_- = h_- + f_+ + g_+.
$$
By the additivity of $T$ from Lemma~\ref{lem:additive}, we can apply $T$ to both sides, yielding
$$
T(h_+) + T(f_-) + T(g_-) = T(h_-) + T(f_+) + T(g_+).
$$
Rearranging terms of the last identity gives
$$
T(h_+) - T(h_-) = (T(f_+) - T(f_-)) + (T(g_+) - T(g_-)).
$$
By definition~\eqref{eq:TildeT_def}, this implies
$\TildeT(h) = \TildeT(f) + \TildeT(g)$.

%--- homogeneity ---

\textbf{Homogeneity}:
We show that $\TildeT(cf) = c\TildeT(f)$ for all $c \in \mathbb{R}$ and $f \in \Czero{X}$.

\textbf{Case 1.} If $c \ge 0$, then $(cf)_+ = c \cdot f_+$ and $(cf)_- = c \cdot f_-$. 
By definition~\eqref{eq:TildeT_def} and the positive homogeneity of $T$
from Lemma~\ref{lem:additive}, we obtain
\[
\TildeT(cf) = T((cf)_+) - T((cf)_-)
            = c T(f_+) - c T(f_-)
            = c \TildeT(f).
\]

\textbf{Case 2.} If $c < 0$, let $r = -c$. Then $r > 0$.
Note that $(-f)_+ = f_-$ and $(-f)_- = f_+$.
Consequently, $(cf)_+ = (-rf)_+ = r \cdot f_-$ and $(cf)_- = (-rf)_- = r \cdot f_+$.
It follows that
\[
\TildeT(cf)
            = T(rf_-) - T(rf_+)
            = r T(f_-) - r T(f_+)
            = -r (T(f_+) - T(f_-))
            = c \TildeT(f).
\]
Thus, $\TildeT$ is homogeneous.
\end{proof}

%\textbf{Homogeneity}:
%We show that $\TildeT(cf) = c\TildeT(f)$ for all $c \in \mathbb{R}$ and $f \in \Czero{X}$.
%\begin{itemize}
%\item Case 1: $c \ge 0$.
%Since $c \ge 0$, we have $(cf)_+ = c \cdot f_+$ and $(cf)_- = c \cdot f_-$. 
%By definition~\eqref{eq:TildeT_def} and the positive homogeneity of $T$
%from Lemma~\ref{lem:additive}, we obtain
%\[
%\TildeT(cf) = T((cf)_+) - T((cf)_-)
%            = c T(f_+) - c T(f_-)
%            = c \TildeT(f).
%\]
%\item Case 2: $c < 0$.
%Let $c = -r$, where $r > 0$. Note that $(-f)_+ = f_-$ and $(-f)_- = f_+$.
%Consequently, $(cf)_+ = (-rf)_+ = r \cdot f_-$ and $(cf)_- = (-rf)_- = r \cdot f_+$.
%It follows that
%\[
%\TildeT(cf)
%            = r T(f_-) - r T(f_+)
%            = -r \TildeT(f)
%            = c \TildeT(f).
%\]
%\end{itemize}
%Thus, $\TildeT$ is homogeneous.
%\end{proof}

%---Lemma 4 導入文

Once the linearity of $\tilde{T}$ has been established, it remains to show that $\tilde{T}$ is bijective.

%--- Lemma4 ---
\begin{lem} \label{lem:bijectivity}
The extension map $\TildeT$ is bijective.
\end{lem}
%--- intectivity ---
\begin{proof}
\textbf{Injectivity}:
Assume $\TildeT(f) = \TildeT(g)$.
By definition~\eqref{eq:TildeT_def}, this means
$T(f_+) - T(f_-) = T(g_+) - T(g_-)$.
By the additivity of $T$ from Lemma~\ref{lem:additive}, we have
$T(f_+ + g_-) = T(g_+ + f_-)$.
Since $T$ is injective, it follows that
$f_+ + g_- = g_+ + f_-$,
which implies $f_+ - f_- = g_+ - g_-$, and therefore $f = g$.
Thus, $\TildeT$ is injective.

%--- surjectivity ---
\textbf{Surjectivity}:
Let $u \in \Czero{Y}$. We write $u = u_+ - u_-$.
Since $T\colon \CzeroPlus{X} \to \CzeroPlus{Y}$ is surjective,
there exist $f_1, f_2 \in \CzeroPlus{X}$ such that
$T(f_1) = u_+$ and $T(f_2) = u_-$.
Define $f = f_1 - f_2$. Then $f \in \Czero{X}$.
By the linearity of $\TildeT$ and the fact that
$\TildeT = T$ on $\CzeroPlus{X}$, we obtain
\[
\TildeT(f) = \TildeT(f_1) - \TildeT(f_2)
           = T(f_1) - T(f_2)
           = u_+ - u_- = u.
\]
Thus, $\TildeT$ is surjective.
\end{proof}

%--- Lemma 5 の導入文---
Next, we recall a standard norm identity for arbitrary $f \in \Czero{X}$.
The identity
\begin{equation}
 \V{f} = \max\{\V{f_+}, \V{f_-}\} \label{eq:f_norm}
\end{equation}
is a standard property of the supremum norm on $\Czero{X}$.
We now show that the extended map $\TildeT$ satisfies an analogous identity.

%---Lemma 5 ---
\begin{lem}
The following identity holds for all $f \in \Czero{X}$:
\begin{equation}
 \V{\TildeT(f)} = \max\{\V{T(f_+)}, \V{T(f_-)}\}. \label{eq:Tf_norm}
\end{equation}
\end{lem}

%--- Lemma 5 証明

\begin{proof}
Let $f \in \Czero{X}$.
From condition~\eqref{eq:2} and Lemma~\ref{lem:additive}, we have
\[
\V{T(f_+) \cdot T(f_-)} = \V{T(f_+ \cdot f_-)} = \V{T(0)} = 0.
\]
This implies $T(f_+)\cdot T(f_-)=0$.
Since $T$ maps $\CzeroPlus{X}$ into $\CzeroPlus{Y}$, we have $T(f_+) \ge 0$ and $T(f_-) \ge 0$.
Thus, for any $y \in Y$, at least one of $T(f_+)(y)$ or $T(f_-)(y)$ is zero.
Consequently,
\begin{align*}
|\TildeT(f)(y)|
&= |T(f_+)(y) - T(f_-)(y)| \\
&= \max\{T(f_+)(y), T(f_-)(y)\}.
\end{align*}
Taking the supremum over all $y \in Y$, we obtain
\[
\V{\TildeT(f)} = \sup_{y \in Y} \max\{T(f_+)(y), T(f_-)(y)\} \le \max\{\V{T(f_+)}, \V{T(f_-)}\}.
\]
To prove the reverse inequality, observe that for any $y \in Y$,
\[
\max\{T(f_+)(y), T(f_-)(y)\} \ge T(f_+)(y),T(f_-)(y).
\]
Taking the supremum over all $y \in Y$, we obtain $\V{\TildeT(f)} \ge \V{T(f_+)},\V{T(f_-)}$.
Therefore,
\[
\V{\TildeT(f)} \ge \max\{\V{T(f_+)}, \V{T(f_-)}\}.
\]
Combining these inequalities, we conclude that
\[
\V{\TildeT(f)} = \max\{\V{T(f_+)}, \V{T(f_-)}\},
\]
which establishes identity~\eqref{eq:Tf_norm}.
\end{proof}

Having established the linearity and bijectivity of $\TildeT$, we now examine the norm of $\TildeT$.

% --- Lemma 6---
\begin{lem} \label{lem:bounded}
The map $\TildeT$ satisfies the inequality
$\V{\TildeT(f)} \le \V{f}$ for all $f \in \Czero{X}$.
\end{lem}

%--- Lemma 6 証明---
\begin{proof}
First, we show that $\V{T(h)} \le \V{h}$ for all $h \in \CzeroPlus{X}$.
Let $h \in \CzeroPlus{X}$. Then $0\le h^2 \le \V{h} h$.
Since $T$ is order-preserving by Lemma \ref{lem:order_preserving} and positively homogeneous by Lemma \ref{lem:additive}, we apply $T$ to the inequality, obtaining
$0\le T(h^2) \le T(\V{h} h) = \V{h} T(h)$.
Taking the norm of both sides yields
$\V{T(h^2)} \le \V{h} \cdot \V{T(h)}$.
By assumption \eqref{eq:2}, $\V{T(h^2)}$ equals $\V{T(h)}^2$.
Substituting this, we obtain
$$
\V{T(h)}^2 \le \V{h} \cdot \V{T(h)} .
$$
If $h \ne 0$, then $T(h) \ne 0$ because $T$ is injective.
We can thus divide the last inequality by $\V{T(h)}$ to obtain
$\V{T(h)} \le \V{h}$.
This inequality holds trivially if $h = 0$ since $T(0)=0$.
Hence, the inequality $\V{T(h)} \le \V{h}$ holds for all $h \in \CzeroPlus{X}$.

Next, we show $\V{\TildeT(f)} \le \V{f}$ for all $f \in \Czero{X}$.
From identity \eqref{eq:Tf_norm}, we have
$\V{\TildeT(f)} = \max\{\V{T(f_+)}, \V{T(f_-)}\}$.
As shown in the first part of the proof, we know that
$\V{T(f_+)} \le \V{f_+}$ and $\V{T(f_-)} \le \V{f_-}$.
Therefore,
$\V{\TildeT(f)} \le \max\{\V{f_+}, \V{f_-}\}$.
Using identity \eqref{eq:f_norm}, we conclude that
$\V{\TildeT(f)} \le \V{f}$.
\end{proof}

%--- Lemma 7 導入文 ---
We have established that $\TildeT$ is a bounded linear bijection between the Banach spaces $\Czero{X}$ and $\Czero{Y}$.
This allows us to apply the Open Mapping Theorem.

% --- Lemma 7 ---
\begin{lem} \label{lem:isometry}
$\TildeT$ is an isometry.
\end{lem}

%--- Lemma 7 証明---
\begin{proof}
Since $\TildeT$ is a bounded linear bijection between the Banach spaces $\Czero{X}$ and $\Czero{Y}$, the Open Mapping Theorem implies that the inverse map $\TildeT^{-1}$ is bounded.
Therefore, there exists a constant $M > 0$ such that
$$
\V{f} \le M\V{\TildeT(f)} \quad (f \in \Czero{X}).
$$

We first show that $T$ preserves the norm on the positive cone $\CzeroPlus{X}$.
Let $h \in \CzeroPlus{X} \setminus \{0\}$.
Setting $g = h / \V{h}$, we have $\V{g} = 1$.
We apply condition \eqref{eq:2} to show that
$\V{T(g^2)} = \V{T(g)}^2$.
By induction, we obtain the following identity for all $n \in \mathbb{N}$:
$$
\V{T(g^{2^n})} = \V{T(g)}^{2^n}.
$$
Let $n \in \mathbb{N}$ be fixed.
It follows from the previous inequality that
$$
1 = \V{g^{2^n}} \le M\V{\TildeT(g^{2^n})}
= M\V{T(g^{2^n})}
= M\V{T(g)}^{2^n}.
$$
Taking the $2^n$-th root on both sides yields
$1 \le M^{1/2^n} \V{T(g)}$.
Letting $n \to \infty$, we obtain $1 \le \V{T(g)}$.
Since $\V{T(g)} = \V{\TildeT(g)} \le \V{g} = 1$
by Lemma \ref{lem:bounded},
we conclude that $\V{T(g)} = 1$.
By the positive homogeneity of $T$ in Lemma \ref{lem:additive},
this implies $\V{T(h)} = \V{h}$ for all $h \in \CzeroPlus{X}$.

Finally, we prove that $\TildeT$ preserves the norm on $\Czero{X}$.
Let $f \in \Czero{X}$ be arbitrary.
Combining identities \eqref{eq:f_norm} and \eqref{eq:Tf_norm}
with $\V{T(h)} = \V{h}$ for $h \in \CzeroPlus{X}$,
we directly obtain:
$$
\V{\TildeT(f)}
= \max\{\V{T(f_+)}, \V{T(f_-)}\}
= \max\{\V{f_+}, \V{f_-}\}
= \V{f}.
$$
Since $\TildeT$ is linear, this equality confirms that $\TildeT$ is an isometry.
\end{proof}

%--- proof of main theoremの導入文
With the properties of the extension map $\TildeT$ established in the preceding lemmas,
we are now ready to prove the main result.

% --- Proof of Main Theorem ---
\begin{proof}[\textbf{Proof of Theorem}]
By Lemma \ref{lem:isometry}, $\TildeT\colon \Czero{X} \to \Czero{Y}$ is a surjective linear isometry.
By the Banach--Stone Theorem \cite{BanachStone, Wang95}, there exist 
a continuous function $\alpha\colon Y \to \{-1, 1\}$ and a homeomorphism $\tau\colon Y \to X$ such that
$$
\TildeT(f)(y) = \alpha(y) f(\tau(y)) \quad (f \in \Czero{X}, \, y \in Y).
$$
Recall that $\TildeT = T$ on $\CzeroPlus{X}$.
We show that $\alpha = 1$ on $Y$.

Let $y_0 \in Y$ be arbitrary.
By Urysohn's lemma, we can choose a function $f_0 \in \CzeroPlus{X}$ such that
$f_0(\tau(y_0)) = 1$.
Applying the previous equality, we obtain
$$ \alpha(y_0)=\alpha(y_0)\cdot 1=\alpha(y_0)f_0(\tau(y_0))
=T(f_0)(y_0)\geq0 ,
$$
because $T(f_0)\in\CzeroPlus{Y}$.
Since $\alpha(y_0) \in \{-1, 1\}$, this forces $\alpha(y_0) = 1$.
Since $y_0$ was arbitrary, we conclude that $\alpha(y) = 1$ for all $y \in Y$.

Thus, the original map $T$ is given by
$T(f) = f \circ \tau$ for all $f \in \CzeroPlus{X}$.
\end{proof}

\begin{ack}
The first author was partially supported by
JSPS KAKENHI Grant Number JP 25K07028.
\end{ack}

% --- Bibliography ---

\end{document}